\documentclass{amsart}
\usepackage[english]{babel}
\usepackage{amssymb}

\numberwithin{equation}{section}
\theoremstyle{plain}
  \newtheorem{teor}{Theorem}[section]
  \newtheorem{prop}[teor]{Proposition}
  \newtheorem{lemma}[teor]{Lemma}
  \newtheorem{cor}[teor]{Corollary}
\theoremstyle{definition}
  
\theoremstyle{remark}

  \newtheorem{oss}[teor]{Remark}

\newcommand{\gra}[1]{\left\{#1\right\}}
\newcommand{\pa}[1]{\left(#1\right)}

\newcommand{\qua}[1]{\left[#1\right]}
\newcommand{\abs}[1]{\left\lvert#1\right\rvert}

\newcommand{\R}{\mathbb{R}}

\newcommand{\N}{\mathbb{N}}

\newcommand{\Q}{\mathbb{Q}}
\def\rn{\mathbb{R}^n}

\newcommand{\bx}{\bar{x}}

\renewcommand{\leq}{\leqslant} 
\renewcommand{\geq}{\geqslant}

\renewcommand{\phi}{\varphi}
\renewcommand{\epsilon}{\varepsilon}

\DeclareMathOperator{\supp}{Supp}

\title[Stability for BBL inequalities]{Stability for Borell-Brascamp-Lieb inequalities}

\begin{document}

\author{Andrea Rossi and Paolo Salani}
\address{Andrea Rossi, DiMaI "U. Dini" - Universit\`a di Firenze, andrea.rossi@unifi.it}
\address{Paolo Salani, DiMaI "U. Dini" - Universit\`a di Firenze, paolo.salani@unifi.it}
\begin{abstract}
We study stability issues for the so-called Borell-Brascamp-Lieb inequalities, proving that when near equality is realized, the involved functions must be $L^1$-close to be $p$-concave and to coincide up to homotheties of their graphs. 
\end{abstract}
\maketitle

\section{Introduction}
The aim of this paper is to study the stability of the so-called {\em Borell-Brascamp-Lieb inequality} (BBL inequality below), which we recall hereafter. 
\begin{prop}[BBL inequality] \label{bbl}
Let $0<\lambda <1, -\frac{1}{n} \leq p \leq +\infty$, $0\leq f,g,h\in L^1(\rn)$ and  assume the following holds
\begin{equation}
h((1-\lambda) x + \lambda y) \geq \mathcal{M}_p(f(x), g(y); \lambda)
\end{equation}
for every $x,\,y\in\rn$. Then
\begin{equation}\label{eq0}
\int_{\mathbb{R}^n} \! h \, dx\geq \mathcal{M}_{\frac{p}{np+1}} \left (\int_{\rn} f\,dx, \int_{\rn} g\,dx\,; \lambda \right).
\end{equation}
\end{prop}
Here the  number $p/(np + 1)$ has to be interpreted in the obvious way in the extremal cases (i.e. it is
equal to $-\infty$ when $p =- 1/n$ and to $1/n$ when $p = +\infty$) and the quantity $\mathcal{M}_q(a,b;\lambda)$ represents the ($\lambda$-weighted) {\em $q$-mean} of two nonnegative numbers $a$ and $b$,  
that is $\mathcal{M}_q(a, b; \lambda)=0$ if $ab=0$ for every $q\in\R\cup\{\pm\infty\}$ and
\begin{equation}\label{pmean}
\mathcal{M}_q(a, b; \lambda) =\left\{
 \begin{array}{ll}
 \max\{a,b\}  \, \,\, \, &\,\,q=+\infty\,, \\
 \left[ (1-\lambda)a^q + \lambda b^q\right]^{\frac{1}{q}} \, \, &\,\,0\neq q\in\R\,, \\
a^{1-\lambda} b^\lambda \,\, \, \, &\,\,q=0\,,\\
  \min\{a,b\} \, \, \, \, &\,\,q=-\infty\,,
                  \end{array}
\right.\,\quad\text{if }ab>0\,.
\end{equation}


The BBL inequality was first proved (in a slightly different form) for $p > 0$  by Henstock and Macbeath (with $n=1$)
in \cite{HM} and by Dinghas in \cite{DI}. Then it was generalized by Brascamp and Lieb in \cite{BL2} and by Borell in \cite{BO}. The case $p = 0$ is usually known as {\em Pr\'{e}kopa-Leindler inequality}, as it was previously proved by Pr\'{e}kopa \cite{Prekopa} and Leindler \cite{Leindler} (later rediscovered by Brascamp and Lieb in \cite{BL1}). 

In this paper we deal only with the case $p>0$ and are particularly interested in the equality conditions of BBL, that are discussed in \cite{Du} (see Theoreme 12 therein). 
To avoid triviality, if not otherwise explicitly declared, we will assume throughout the paper that $f,\,g\in L^1(\rn)$ are nonnegative compactly supported functions (with supports $\supp(f)$ and $\supp(g)$) such that
$$
F=\int_{\rn}f\,dx>0\quad\text{and}\quad G=\int_{\rn}g\,dx>0\,.
$$

Let us restate a version of the BBL inequality including its equality condition in the case 
$$p=\frac1s>0\,,$$ 
adopting a slightly different notation.
 \begin{prop}\label{bbleq}
Let $s>0$ and $f,g$ be as said above. Let $\lambda\in(0,1)$ and $h$ be a nonnegative function belonging to $L^1(\rn)$ such that
\begin{equation}\label{assumptionh}
h((1-\lambda)x+\lambda y)\geq \left((1-\lambda)f(x)^{1/s}+\lambda g(y)^{1/s}\right)^s
\end{equation}
for every $x\in\supp(f)$, $y\in\supp(g)$.

Then
\begin{equation}\label{bbls}
\int_{\rn} h\ dx  \geq \mathcal{M}_{\frac{1}{n+s}}\left(F,G;\lambda\right)\,.
\end{equation}
Moreover equality holds in \eqref{bbls} only if there exists a nonnegative concave function $\phi$ such that
\begin{equation}\label{equalityBBLeq}
\phi(x)^s=a_1\,f(b_1x-\bar x_1)=a_2\,g(b_2x-\bar x_2)=a_3\,h(b_3x-\bar x_3)\quad\text{a.e. }x\in\rn\,,
\end{equation}
for some $\bar x_1,\bar x_2,\bar x_3\in\rn$ and suitable $a_i,b_i>0$ for $i=1,2,3$.
\end{prop}
Notice that, given $f$ and $g$, the smallest function satisfying \eqref{assumptionh} (hence the smallest 
function to which Proposition \ref{bbleq} possibly applies to) is 
their  $p$-Minkowksi sum (or $(p,\lambda)$-supremal convolution), defined as follows (for $p= \frac{1}{s}$)
\begin{equation}\label{hdef}
h_{s,\lambda}(z) = \sup\gra{\left((1-\lambda)f(x)^{1/s}+\lambda g(y)^{1/s}\right)^s\,:\,  z=(1-\lambda)x +\lambda y}
\end{equation}
for $z\in(1-\lambda)\supp(f)+\lambda\supp(g)$ and $h_{s,\lambda}(z)=0$ if $z\notin(1-\lambda)\supp(f)+\lambda\supp(g)$.
\medskip

When dealing with a rigid inequality, a natural question arises about the stability of the equality case; here the question at hand is the following: if we are close to equality in \eqref{bbls}, must the functions $f,\,g$ and $h$ be close (in some suitable sense) to satisfy \eqref{equalityBBLeq}?

The investigation of stability issues in the case $p=0$ was started by Ball and B\"{o}r\"{o}czky in \cite{BB1, BB2} and new related results are in \cite{BF}. The general case $p>0$ has been very recently faced in \cite{GS}. But the results of \cite{GS}, as well as the quoted results for $p=0$, hold only in the restricted class of $p$-concave functions, hence answering only a half of the question.
Let us recall here the definition of $p$-concave function: a nonnegative function $u$ is {\em $p$-concave} for some $p\in\R\cup\{\pm\infty\}$ if
$$
u((1-\lambda)x+\lambda y)\geq\mathcal{M}_p(u(x), u(y); \lambda)\quad\text{for every }x,y\in\rn\text{ and  every }\lambda\in(0,1)\,.
$$
Roughly speaking, $u$ is $p$-concave if it has convex support $\Omega$ and: (i) $u^p$ is concave in $\Omega$ for $p>0$; (ii) $\log u$ is concave in $\Omega$ for $p=0$; (iii) $u^p$ is convex in $\Omega$ for $p<0$; (iv) $u$ is quasi-concave, i.e. all its superlevel sets are convex, for $p=-\infty$; (v) $u$ is a positive constant in $\Omega$, for $p=+\infty$.

Here we want to remove this restriction, proving that near equality in \eqref{bbls} is possible if and only if the involved functions are close to coincide up to homotheties of their graphs and they are also nearly $p$-concave, in a suitable sense. 
But before stating our main result in detail, we need to introduce some notation: for $s>0$, we say that two functions $v,\hat{v}:\rn\to[0,+\infty)$ are  {\em $s$-equivalent}  if there exist $\mu_v>0$ and $\bar x\in\rn$ such that 
\begin{eqnarray}
\hat{v}(x)= \mu_v^s  \,v\pa{\frac{x-\bar x}{\mu_v}}\qquad \text{a.e. }x\in\rn.\label{s-equivalent}
\end{eqnarray}
Now we are ready to state our main result, which regards the case $s=1/p\in\N$. Later (see \S4) we will extend the result to the case $0<s\in\Q$ in Corollary \ref{corstab} and finally (see Corollary \ref{stabparteintera} in \S5) we will give a slightly weaker version, valid for every $s>0$.
\begin{teor}\label{stab1}
Let $f,g,h$ as in Proposition \ref{bbleq} with
$$0<s\in\N\,.
$$ 
Assume that \begin{eqnarray}  
\int_{\R^n} h\ dx \leq \mathcal{M}_{\frac{1}{n+s}} \left(F, G\,; \lambda \right)
+ \epsilon\label{ipstab}
\end{eqnarray}
for some $\epsilon > 0 $ small enough. 

Then there exist a $\frac1s$-concave function $u: \R^n \longrightarrow [0,+\infty)$ and two functions $\hat{f}$ and $\hat{g}$, 
$s$-equivalent to $f$ and $g$ in the sense of \eqref{s-equivalent} 
(with suitable $\mu_f$ and $\mu_g$ given in \eqref{mufmug})
such that the following hold:
\begin{equation}\label{tesistab0} 
 \qquad  u \geq \hat{f},   \qquad  \qquad  u \geq \hat{g}\,, 
\end{equation}
\begin{equation} \int_{\R^n} (u-\hat{f}) \ dx \ + \ \int_{\R^n} \pa{u-\hat{g}} \ dx 
 \  \leq \ C_{n+s}\pa{\frac{\epsilon}{\mathcal{M}_{\frac{1}{n+s}} \left(F, G\,; \lambda \right)}}\,, \label{tesistab}
\end{equation}
where $C_{n+s}(\eta)$ is an infinitesimal function for $\eta \longrightarrow 0$ (whose explicit expression is
given later, see \eqref{explicitCns}).
\end{teor}
Notice that the function $u$ is bounded, hence as a byproduct of the proof we obtain that the functions $f$ 
and $g$ have to be bounded as well (see Remark \ref{rem1}).\\
The proof of the above theorem is based on a proof of the BBL inequality due to Klartag \cite{Kla}, which directly connects the BBL inequality to the Brunn-Minkowski inequality, and the consequent application of a recent stability result for the Brunn-Minkowski inequality by Figalli and Jerison \cite{FJ}, which does not require any convexity assumption of the involved sets. Indeed \cite{FJ} is the first paper, at our knowledge, investigating on stability issues for the Brunn-Minkowski inequality outside the realm of convex bodies. Noticeably, Figalli and Jerison ask therein for a functional counterpart of their result, pointing out that {\em "at the moment some stability
estimates are known for the  Pr\'{e}kopa-Leindler inequality only in one dimension or for some special class of functions \cite{BB1, BB2},
and a general stability result would be an important direction of future investigations."}.  Since BBL inequality is the functional counterpart of the Brunn-Minkowksi inequality (for any $p>0$ as much as for $p=0$), this paper can be considered a first answer to the question by Figalli and Jerison.
\medskip

The paper is organized as follows.
The Brunn-Minkowski inequality and the stability result of \cite{FJ} are recalled in \S2, where we also discuss the equivalence between the Brunn-Minkowski and the BBL inequality. In \S3 we prove Theorem \ref{stab1}. 
Finally \S4 contains the already mentioned generalization to the case of rational $s$, namely Corollary \ref{corstab}, while \S5  is devoted to Corollary \ref{stabparteintera}, where we prove a stability for every $s>0$ under a suitable normalization for $\int f$ and $\int g$.
The paper ends with an Appendix (\S6) where we give the proofs of some easy technical lemmas
for the reader's convenience.

\bigskip

{\bf Acknowledgements.} The second author has been partially supported
by INdAM in the framework of a GNAMPA project, and by MIUR in the framework of a PRIN 2013 project and a FIR 2013 project.

\section{Preliminaries}
\subsection{Notation} Throughout the paper the symbol $|\cdot|$ is used to denote different things and we hope this is not going to cause confusion. In particular: for a real number $a$ we denote by $|a|$ its absolute value, as usual; for a vector $x=(x_1,\dots,x_m)\in\R^m$ we denote by $|x|$ its euclidean norm, that is $|x|=\sqrt{x_1^2+\dots+x_m^2}$; for a set $A\subset\R^m$ we denote by $|A|$ its ($m$-dimensional) Lebesgue measure or, sometimes, its outer measure if $A$ is not measurable.

The support set of a nonnegative function $f:\R^m\to[0,+\infty)$ is denoted by $\supp(f)$, that is $\supp(f)=\overline{\{x\in\R^m\,:\,f(x)>0\}}$.

Let $\lambda \in (0,1)$, the Minkowski convex combination (of coefficient $\lambda$)
of two nonempty sets $A,B \subseteq \R^n$ is given by
\[(1-\lambda)A + \lambda B = 
\gra{(1-\lambda)a + \lambda b:\ a \in A,\ b \in B}.\]

\subsection{About the Brunn-Minkowski inequality}
The classical form of the Brunn-Minkowski inequality (BM in the following) regards only convex bodies and it is at the core of the related theory (see \cite{sch}). Its validity has been extended later to the class of measurable sets and we refer to the beautiful paper by Gardner \cite{gardner} for a throughout presentation of BM inequality, its history and its intriguing relationships with many other important geometric and analytic inequalities. 
Let us now recall it (in its general form).
\begin{prop}[Brunn-Minkowski inequality]\label{BM}
Given $\lambda \in (0,1)$, let $A, B \subseteq \R^n$ be nonempty measurable
sets. 
Then
\begin{equation}
\abs{(1-\lambda)A + \lambda B}^{1/n} \geq  (1-\lambda)\abs{A}^{1/n} + \lambda \abs{B}^{1/n}\, \label{tesiBM}
\end{equation}
(where $|\cdot|$ possibly means outer measure if $(1-\lambda)A + \lambda B$ is not measurable).

In addition, if $\abs{A},\abs{B}>0$, then equality in \eqref{tesiBM} holds if and only if
there exist a convex set $K\subseteq \R^n$, $v_1,v_2\in \R^n$ and $\lambda_1,\lambda_2 >0$ such that
\begin{equation}
\lambda_1 A + v_1 \subseteq K, \quad \lambda_2 B + v_2 \subseteq K, \quad   
\abs{K\setminus\pa{\lambda_1 A + v_1}}= \abs{K\setminus\pa{\lambda_2 B + v_2}} =0.  \label{uguaBM}
\end{equation}
\end{prop} 
We remark that equality holds in \eqref{tesiBM} if and only if the involved sets are convex (up to a null measure set) and homothetic.

The stability of BM inequality was first investigated only in the class of convex sets, see for instance \cite{Diskant, EK, Gro, FMP1, FMP2, Segal}. Very recently Christ  \cite{Ch1, Ch2} started the investigation without convexity assumptions, and its qualitative results have been made quantitative and sharpened by Figalli and Jerison in \cite{FJ}; here is their result, for $n\geq 2$.
\begin{prop}\label{FJ}
Let $n\geq 2,$ and $A,B\subset \R^n$ be 
measurable sets with $|A|=|B|=1$. Let $\lambda\in(0,1)$, set $\tau= \min\gra{\lambda, 1-\lambda}$ and 
$S=(1-\lambda)A +\lambda B$. If
 \begin{eqnarray}
{\abs{S}} \ \leq \ 1+\delta  \label{ipotFJ}
\end{eqnarray}
for some $\delta\leq e^{-M_n(\tau)}$, then there exists a convex $K\subset \R^n$ such that, up to a translation,
\[  A,B \subseteq K  \qquad \text{and} \qquad 
\abs{K\setminus A} + \abs{K\setminus B} \leq \tau^{-N_n} \delta^{\sigma_n(\tau)}.\]
The constant $N_n$ can be explicitly computed and we can take
\[  M_n(\tau)= \frac{2^{3^{n+2}} n^{3^n} \abs{\log \tau}^{3^n}}{\tau^{3^n}},
\qquad  \sigma_n(\tau)= \frac{\tau^{3^n}}{2^{3^{n+1}} n^{3^n} \abs{\log \tau}^{3^n}}.   \]
\end{prop}

\begin{oss} As already said, the proof of our main result is based on Proposition \ref{FJ} and now we can give the explicit expression of the infinitesimal function $C_{n+s}$ of Theorem \ref{stab1}: 
\begin{equation}\label{explicitCns}C_{n+s}(\eta)=\frac{\eta^{\sigma_{n+s}(\tau)}}{\omega_s\,  \tau^{N_{n+s}}\,,}\,,
\end{equation}
where $\omega_s$ denotes the measure of the unit ball in $\R^s$.
\end{oss}

Next, for further use, we rewrite Proposition \ref{FJ} without the normalization constraint about the measures of the involved sets $A$ and $B$.
\begin{cor}\label{corFJ}
Let $n\geq 2$ and $A,B\subset \R^n$ be 
measurable sets with $\abs{A},\abs{B}\in(0,+\infty)$. Let $\lambda\in(0,1)$, set $\tau= \min\gra{\lambda, 1-\lambda}$ and
$S=(1-\lambda)A +\lambda B$. 
If \begin{eqnarray}
 \frac{\abs{S} - \qua{(1-\lambda)\abs{A}^{1/n} + \lambda \abs{B}^{1/n}}^n}
 {\qua{(1-\lambda)
 \abs{A}^{1/n} + \lambda\abs{B}^{1/n}}^n}
  \leq \delta  \label{ipotcorFJ}
  \end{eqnarray}
for some $\delta \leq e^{-M_n(\tau)}$, then there exist a convex $K\subset \R^n$  and two homothetic copies $\tilde{A}$ and $\tilde{B}$ of $A$ and $B$ such that
\[  \tilde{A},\tilde{B} \subseteq K  \qquad \text{and} \qquad 
\abs{K\setminus \tilde{A}} + \abs{K\setminus\tilde{B}} \leq \tau^{-N_n} \delta^{\sigma_n(\tau)}.\]
\end{cor}


\begin{proof}
The proof is standard and we give it just for the sake of completeness. 
First we set
\[ \tilde{A}=  \frac{A}{\abs{A}^{1/n}}, \qquad  \tilde{B}=  \frac{B}{\abs{B}^{1/n}}\]
so that $|\tilde{A}|=|\tilde{B}|=1$. Then we define
\[ \tilde{S} := \mu \tilde{A}+ (1-\mu)\tilde{B} \qquad \text{with} \quad
\mu= \frac{(1-\lambda)\abs{A}^{1/n}}{(1-\lambda)\abs{A}^{1/n} + \lambda \abs{B}^{1/n}}\,,\]
and observe that $|\tilde{S}|\geq 1$ by the Brunn-Minkowski inequality.
It is easily seen that
 \[\tilde{S} =  \frac{S}{(1-\lambda)\abs{A}^{1/n} + \lambda \abs{B}^{1/n}}\,.\]
Now we see that the hypothesis \eqref{ipotFJ} holds for $\tilde{A}, \tilde{B},\tilde{S}$, indeed
\[ \abs{\tilde{S}}-1 = 
 \frac{\abs{S}-\qua{(1-\lambda)\abs{A}^{1/n} + \lambda \abs{B}^{1/n}}^n}
 {\qua{(1-\lambda)\abs{A}^{1/n} + \lambda \abs{B}^{1/n}}^n}
 \leq \delta, \]
by \eqref{ipotcorFJ}.
Finally Proposition \ref{FJ} applied to $\tilde{A}, \tilde{B}$ and $\tilde{S}$ implies the result and
this concludes the proof.
\end{proof}

\subsection{The equivalence between BBL and BM inequalities}
The equivalence between the two inequalities is well known
and it becomes apparent as soon as one notices that the $(p,\lambda)$-supremal convolution defined in \eqref{hdef} corresponds
to the Minkowski linear combinations of the graphs of $f^{p}$ and $g^{p}$. In particular, for $p=1$, \eqref{eq0} coincides with
\eqref{tesiBM} where $A=\{(x,t)\in\R^{n+1}\,:\,0\leq t\leq f(x)\}$ and $B=\{(x,t)\in\R^{n+1}\,:\,0\leq t\leq g(x)\}$.

To be precise, that Proposition \ref{bbl} implies \eqref{tesiBM} is easily seen by applying \eqref{eq0}
to the case $f=\chi_A$, $g=\chi_B$, 
$h=\chi_{(1-\lambda)A+ \lambda B}$, $p=+\infty$.
The opposite implication can be proved in several ways; hereafter we present a proof due to Klartag \cite{Kla}, which
is particularly useful for our goals.\\

To begin, given two integers $n,s>0$, 
let $f: \R^n \longrightarrow [0,+\infty)$
be an integrable function 
with nonempty support
(to avoid the trivial case in which f is identically zero).
Following Klartag's notations and ideas \cite{Kla} (see also \cite{AKM}), we associate with $f$
the nonempty measurable set
\begin{eqnarray}
 K_{f,s} = \gra{(x,y)\in \R^{n+s}= \R^n \times \R^s: \ x \in \supp(f),\ \abs{y}\leq f(x)^{1/s}}, \label{solido}
\end{eqnarray}
where obviously $x\in \R^n$ and $y \in \R^s$.
In other words, $K_{f,s}$ is the subset of $\R^{n+s}$ obtained as union of the $s$-dimensional
closed balls of center $(x,0)$ and radius $f(x)^{1/s}$, for $x$ belonging to the support of $f$, or, if you prefer, the set in $\R^{n+s}$ obtained by rotating with respect to $y=0$ the $(n+1)$-dimensional set $\{(x,y)\in\R^{n+s}\,: 0\leq y_1\leq f(x)^{1/s},\,y_2=\dots=y_s=0\}$.\\
We observe that $K_{f,s}$ is convex if and only if $f$ is $(1/s)$-concave
(that is for us a function $f$ having compact convex support such that 
$f^{1/s}$ is concave on $\supp\pa{f}$). 
If $\supp(f)$ is compact, then $K_{f,s}$ is bounded if and only if $f$ is bounded.\\
Moreover, thanks to Fubini's Theorem, it holds
\begin{equation}
\abs{K_{f,s}} = \int_{\supp\pa{f}} \omega_s \cdot \pa{f(x)^{1/s}}^s\ dx 
= \omega_s \int_{\R^n} f(x)\ dx.
  \label{volK}
\end{equation}
In this way, the integral of $f$ coincides, up to the constant $\omega_s$,
with the volume of $K_{f,s}$. Now we will use this simple identity to prove Proposition \ref{bbleq} as a direct application of the BM inequality.

Although of course the set $K_{f,s}$ depends heavily on $s$, for simplicity from now on we will remove the subindex $s$ and just write $K_f$ for $K_{f,s}$. 

\medskip

Let us start with the simplest case, when $p=1/s$ with $s$ positive integer. 

\begin{prop}[BBL, case $1/p=s\in\N$]\label{BBL}
Let $n, s$ be positive integers, $\lambda \in (0,1)$ and
$f,g,h :\R^n \longrightarrow [0,+\infty)$ be integrable functions, with $\int f>0$ and $\int g>0$.
Assume that for any $x_0\in\supp(f),\ x_1 \in\supp(g)$
\begin{eqnarray}
h\pa{(1-\lambda)x_0 + \lambda x_1} \geq \qua{(1-\lambda)f(x_0)^{1/s} + \lambda g(x_1)^{1/s}}^s.  \label{ipotBBL}
\end{eqnarray}
Then
\begin{eqnarray}
 \pa{\int_{\R^n} h\ dx}^{\frac{1}{n+s}} \geq  
(1-\lambda) \pa{\int_{\R^n} f\ dx}^{\frac{1}{n+s}} +
\lambda \pa{\int_{\R^n} g\ dx}^{\frac{1}{n+s}}. \label{tesiBBL} 
\end{eqnarray}
\end{prop}

\vspace{0.2 cm}
\begin{proof}
Since the integrals of $f$ and $g$ are positive, the sets $K_f$ and $K_g$ 
have positive measure. Let $\Omega_\lambda$ 
be the Minkowski convex combination (with coefficient $\lambda$)
of $\Omega_0=\supp(f)$ and $\Omega_1=\supp(g)$.
Now 
consider the function $h_{s,\lambda}$ as defined by \eqref{hdef}; to simplify the notation, we will denote $h_{s,\lambda}$ 
by $h_\lambda$ from now on.
First notice that the support of $h_\lambda$ is $\Omega_\lambda$. Then
it is easily seen that
\begin{equation}\label{utile}
K_{h_\lambda} = (1-\lambda)K_f + \lambda K_g\,.
\end{equation}
Moreover, since $h\geq h_\lambda$ by assumption \eqref{ipotBBL}, we have
\begin{equation}\label{inclusion}
K_h\supseteq K_{h_\lambda} \,.
\end{equation}
By applying Proposition \ref{BM} to $K_{h_\lambda}, K_f, K_g$ we get
\begin{eqnarray}\label{K} \abs{K_h}^{\frac{1}{n+s}}\geq\abs{K_{h_\lambda}}^{\frac{1}{n+s}} \geq  (1-\lambda)\abs{K_f}^{\frac{1}{n+s}} + \lambda \abs{K_g}^{\frac{1}{n+s}},
\end{eqnarray}
where $\abs{K_{h_\lambda}}$ possibly means the outer measure of the set $K_{h_\lambda}$.\\
Finally \eqref{volK} yields
\[\abs{K_h} 
= \omega_s \int_{\R^n} h\ dx, \qquad
\abs{K_f} =  \omega_s \int_{\R^n} f\ dx , \qquad  
\abs{K_g} =  \omega_s \int_{\R^n} g\ dx ,
\]
thus dividing \eqref{K} by $\omega_s^{\frac{1}{n+s}}$ we get \eqref{tesiBBL}.
\end{proof}

\vspace{0.5 cm}
Next we show how it is possible to generalize Proposition \ref{BBL}
to a positive rational index $s$. The idea is to apply again the
Brunn-Minkowski inequality to sets that generalize 
those of the type \eqref{solido}.
What follows is a slight variant of the proof of Theorem 2.1 in \cite{Kla}.

The case of a positive rational index $s$ requires the following definition.
Given $f: \R^n \longrightarrow [0,+\infty)$ integrable and a positive integer $q$ 
(it will be the denominator of the rational $s$)
we consider the auxiliary function $\tilde{f}: \R^{nq} \longrightarrow [0,+\infty)$
defined as
\begin{eqnarray}
\tilde{f}(x)= \tilde{f}(x_1,...,x_q)= \prod_{j=1}^q f(x_j), \label{fprod}
\end{eqnarray}
where $x= (x_1,...,x_q) \in \pa{\R^n}^q$. We observe that, by construction, 
\begin{eqnarray}
\int_{\R^{nq}} \tilde{f} \ dx =  \pa{\int_{\R^n} f \ dx}^q;    \label{intprod}
\end{eqnarray}
moreover $\quad \supp{\tilde{f}}= \pa{\supp{f}}\times ... \times \pa{\supp{f}}= \pa{\supp{f}}^{q}.$\\ \\
As just done, from now on we write $A^q$ to indicate the Cartesian product
of q copies of a set $A$. 
\begin{oss}\label{osscartesiano} 
Let $A,B$ be nonempty sets, $q>0$ be an integer, $\mu$  a real.
Clearly
\[ \pa{A + B}^q =  A^q + B^q, \qquad  \qquad   \pa{\mu A}^q = \mu A^q.\]
\end{oss}
\vspace{0.2 cm}
To compare products of real numbers of the type \eqref{fprod}
the following lemma is useful. It's a consequence
of H\"older's inequality (see \cite{HLP}, Theorem 10)
for families of real numbers (in our case for two sets of q positive numbers).

\begin{lemma} \label{Holder}
Given an integer $q>0$, let
$\gra{a_1,...,a_q},\ \gra{b_1,...,b_q}$ be two sets of $q$ real numbers.
Then  \[ \abs{\prod_{j=1}^q a_j} + \abs{\prod_{j=1}^q b_j}
\leq  \qua{\prod_{j=1}^q \pa{\abs{a_j}^q + \abs{b_j}^q}}^{1/q}.\]
\end{lemma}
From this lemma we deduce the following.
\begin{cor}\label{corHolder} Let $\lambda \in (0,1), \ s = \frac{p}{q}$ with integers $p,q>0$.\\
Given $f,g: \R^n \longrightarrow [0,+\infty), \ x_1,...,x_q,x'_1,...,x'_q \in \R^n,$ it holds 
\[ (1-\lambda) \prod_{j=1}^q f(x_j)^{1/p} + \lambda \prod_{j=1}^q g(x'_j)^{1/p}
\leq \prod_{j=1}^q \qua{(1-\lambda)f(x_j)^{1/s} + \lambda g(x'_j)^{1/s}}^{1/q}.\]
\end{cor}
\begin{proof} 
Observing that
\[ (1-\lambda) \prod_{j=1}^q f(x_j)^{1/p} + \lambda \prod_{j=1}^q g(x'_j)^{1/p}=
\prod_{j=1}^q  (1-\lambda)^{1/q} f(x_j)^{1/p} +  \prod_{j=1}^q \lambda^{1/q}  g(x'_j)^{1/p},\]
the result follows directly from Lemma \ref{Holder} 
applied to 
$\gra{a_1,...,a_q},\ \gra{b_1,...,b_q}$
with  \[  a_j = (1-\lambda)^{1/q} f(x_j)^{1/p}, \qquad
 b_j = \lambda^{1/q} g(x'_j)^{1/p},  \qquad j=1,...,q .\]
\end{proof}

Let $$s = \frac{p}{q}$$ with integers $p,q>0$ that we can assume are coprime.\\
Given an integrable function $f: \R^n \longrightarrow [0,+\infty)$ not identically zero,
we define the nonempty measurable subset of $\R^{nq+p}$  
\begin{equation}
W_{f,s}=K_{\tilde{f},p} = \gra{(x,y)\in \pa{\R^{n}}^{q} \times \R^p: \ 
 x \in \supp(\tilde{f}),\ 
 \abs{y}\leq \tilde{f}(x)^{1/p}}= \label{solidogen}
\end{equation}
\[ \gra{(x_1,...,x_q,y)\in \pa{\R^{n}}^{q} \times \R^p: \ 
 x_j \in \supp(f) \ \forall\ j=1,...,q,\quad    \abs{y}\leq \prod_{j=1}^q f(x_j)^{1/p}}.\]
We notice that this definition naturally generalizes \eqref{solido},
since in the case of an integer $s>0$ it holds $s=p,\ q=1$, so
in this case $\tilde{f}=f$ and $W_{f,s} =K_f.$

As for $K_{f,s}$, for simplicity we will remove systematically the subindex $s$ and write $W_f$ in place of $W_{f,s}$ if there is no possibility of confusion.
Clearly \begin{equation}
\abs{W_f} = \int_{\supp\pa{\tilde{f}}} \omega_p \cdot \pa{\tilde{f}(x)^{1/p}}^p\ dx = 
 \omega_p \int_{\R^{nq}} \tilde{f}(x)\ dx = \omega_p \pa{\int_{\R^n} f(x)\ dx}^q
  \label{volW}
\end{equation}
where 
the last equality is given by \eqref{intprod}.\\
Moreover we see that $W_f$ is convex if and only if $\tilde{f}$ is $\frac{1}{p}$-concave (that is, if and only if $f$ is $\frac1s$-concave, see Lemma \ref{eredconc} later on). Next we set
\begin{eqnarray}
W= (1-\lambda)W_f + \lambda W_g  \label{W}\,.
\end{eqnarray} 
Finally, we notice that, by \eqref{utile}, we have
$$
W=K_{\tilde{h}_{p,\lambda},p}\,,
$$
where $\tilde{h}_{p,\lambda}$ is the $(1/p,\lambda)$-supremal convolution of $\tilde{f}$ and $\tilde{g}$ as defined in \eqref{hdef}.
In other words,
$W$ is the set made by 
the elements  $(z,y)\in \pa{\R^{n}}^{q} \times \R^p$ such that $z \in (1-\lambda)\supp(\tilde{f})+ \lambda\supp(\tilde{g})$ and
\begin{equation}\begin{array}{rl}
 \abs{y}\leq    &\sup\big\{(1-\lambda) \tilde{f}(x)^{1/p} + \lambda \tilde{g}(x')^{1/p}:\\
&\qquad\qquad z=(1-\lambda)x +\lambda x', 
x \in \supp(\tilde{f}), x'\in \supp(\tilde{g})\big\}.\end{array} \label{tildeinv}
\end{equation}

\begin{lemma}\label{incluW}
With the notations introduced above, it holds
\[ W \subseteq  W_{{h}_\lambda}\subseteq W_h\,,\]
where ${{h}_\lambda}$ is the $(1/s,\lambda)$-supremal convolution of ${f}$,${g}$,
and $h$ is as in Proposition \ref{bbleq}.
\end{lemma}
\begin{proof} 
The second inclusion is obvious, since $h\geq h_\lambda$ by assumption \eqref{assumptionh}.
Regarding the other inclusion, first we notice that \eqref{solidogen} and Remark \ref{osscartesiano} yield
\[  W_{{h}_\lambda}= \gra{(z,y)\in \pa{\R^{n}}^{q} \times \R^p: \ 
 z \in \supp(\tilde{h_\lambda}),\    
 \abs{y}\leq \tilde{h_\lambda}(z)^{1/p}}= \]
 \[ = \gra{(z,y)\in \pa{\R^{n}}^{q} \times \R^p: \ 
 z \in \pa{(1-\lambda)\supp(f)+ \lambda\supp(g)}^q,\    
 \abs{y}\leq \tilde{h_\lambda}(z)^{1/p}}= \]
 \[ = \gra{(z,y)\in \pa{\R^{n}}^{q} \times \R^p: \ 
 z \in (1-\lambda)\supp(\tilde{f})+ \lambda\supp(\tilde{g}),\    
 \abs{y}\leq \tilde{h_\lambda}(z)^{1/p}}\,, \]
 where $\tilde{h_\lambda}$ is the function associated to ${h_\lambda}$ by \eqref{fprod}.
To conclude it is sufficient to compare this with the condition given by \eqref{tildeinv}. \\ 
For every $z\in(1-\lambda)\supp(\tilde{f})+ \lambda\supp(\tilde{g})$ consider
\[\sup\gra{(1-\lambda) \tilde{f}(x)^{1/p} + \lambda \tilde{g}(x')^{1/p}}
= \sup\gra{(1-\lambda) \prod_{j=1}^q f(x_j)^{1/p} + \lambda \prod_{j=1}^q g(x'_j)^{1/p}}\,,\]
where the supremum is made with respect to $x \in \supp(\tilde{f}),\ x'\in \supp(\tilde{g})$ such that $z=(1-\lambda)x +\lambda x'$.
Corollary \ref{corHolder} then implies
\[ \sup\gra{(1-\lambda) \tilde{f}(x)^{1/p} + \lambda \tilde{g}(x')^{1/p}} 
\leq \sup\gra{\prod_{j=1}^q \qua{(1-\lambda)f(x_j)^{1/s} + \lambda g(x'_j)^{1/s}}^{1/q}}\leq \]
\[ \leq \prod_{j=1}^q \gra{\sup{\qua{(1-\lambda)f(x_j)^{1/s} + \lambda g(x'_j)^{1/s}}^{1/q}}}=
\prod_{j=1}^q  \gra{h_\lambda\pa{(1-\lambda)x_j +\lambda x'_j}^{1/qs}}= \]
\[ = \tilde{h_\lambda}\pa{(1-\lambda)x +\lambda x'}^{1/p} =
\tilde{h_\lambda}(z)^{1/p},\]
having used the definition \eqref{fprod} in the penultimate equality.
Therefore if $$\abs{y}\leq    \sup\gra{(1-\lambda) \tilde{f}(x)^{1/p} + \lambda \tilde{g}(x')^{1/p}}\,,$$ that is if $(z,y) \in W$ by \eqref{tildeinv}, 
then $$\abs{y}\leq \tilde{h_\lambda}(z)^{1/p}\,,$$ i.e. $(z,y) \in W_{{h}_\lambda}$. This concludes the proof.
\end{proof}
\vspace{0.2 cm}
We are ready to prove 
the following version of the Borell-Brascamp-Lieb inequality,
which holds for any positive real index $s$ (and in fact also for $s=0$).

\begin{prop}[BBL for $p>0$]\label{BBL2}
Let $s>0,\ \lambda \in (0,1)$, let $n>0$ be integer. 
Given $f,g,h :\R^n \longrightarrow [0,+\infty)$ 
integrable such that $\int f>0$ and $\int g>0$,
assume that for any $x_0\in\supp(f),\ x_1 \in\supp(g)$
\begin{eqnarray}
h\pa{(1-\lambda)x_0 + \lambda x_1} \geq \qua{(1-\lambda)f(x_0)^{1/s} + \lambda g(x_1)^{1/s}}^s.  \label{ipotBBL2}
\end{eqnarray}
Then
\begin{eqnarray}
 \pa{\int_{\R^n} h\ dx}^{\frac{1}{n+s}} \geq  
(1-\lambda) \pa{\int_{\R^n} f\ dx}^{\frac{1}{n+s}} +
\lambda \pa{\int_{\R^n} g\ dx}^{\frac{1}{n+s}}. \label{tesiBBL2} 
\end{eqnarray}
\end{prop}

\vspace{0.2 cm}

\begin{proof}
Assume first that $s>0$ is rational and let $s = \frac{p}{q}$ with $p,q$ coprime positive integers.
Thanks to \eqref{W} we can apply Proposition \ref{BM} to 
$W_{f},\,W_{g}$
(that are nonempty measurable subsets of $\R^{nq+p}$), so
\[ \abs{W}^{\frac{1}{nq+p}} \geq  (1-\lambda)\abs{W_f}^{\frac{1}{nq+p}} + \lambda \abs{W_g}^{\frac{1}{nq+p}},\]
where $\abs{W}$ possibly means the outer measure of the set $W$.
On the other hand Lemma \ref{incluW} implies $ \abs{W_h}  \geq \abs{W}$, thus
\[ \abs{W_h}^{\frac{1}{nq+p}} \geq  (1-\lambda)\abs{W_f}^{\frac{1}{nq+p}} + \lambda \abs{W_g}^{\frac{1}{nq+p}}.\]
Finally the latter inequality with the identity \eqref{volW} is equivalent to
\[ \omega_p ^{\frac{1}{nq+p}}  \pa{ \int_{\R^n} h\ dx}^\frac{q}{nq+p} \geq 
\omega_p ^{\frac{1}{nq+p}} \qua{(1-\lambda)\pa{ \int_{\R^n} f\ dx}^\frac{q}{nq+p}
+\lambda\pa{ \int_{\R^n} g\ dx}^\frac{q}{nq+p}}. \]
Dividing by $\omega_p ^{\frac{1}{nq+p}}$ we get \eqref{tesiBBL2},
since \[\frac{q}{nq+p} = \frac{q}{q(n+s)}= \frac{1}{n+s}\] 
is exactly the required index. 
The case of a real $s>0$ (and also $s=0$) follows by a standard approximation argument.
\end{proof}
\vspace{0.2 cm}

\section{The proof of Theorem \ref{stab1}}
The idea is to apply the result of Figalli-Jerison, more precisely Corollary \ref{corFJ}, to the sets $K_{h_\lambda}, K_f, K_g$, and then translate the result in terms of the involved functions. We remember that with $h_\lambda$ we denote the function $h_{s,\lambda}$ given by \eqref{hdef}. We also recall that we set $F=\int f$ and $G=\int g$.
Thanks to \eqref{volK}, assumption \eqref{ipstab} is equivalent to
\[\omega_s^{-1}   \abs{K_{h}} \leq  \omega_s^{-1} 
\qua{(1-\lambda)\abs{K_f}^{\frac{1}{n+s}}  +\lambda\abs{K_g}^{\frac{1}{n+s}}}^{n+s}+ \epsilon\,, \qquad\]
which, by \eqref{inclusion}, implies
\begin{eqnarray}
 \abs{K_{h_\lambda}} \leq  
\qua{(1-\lambda)\abs{K_f}^{\frac{1}{n+s}}  +\lambda\abs{K_g}^{\frac{1}{n+s}}}^{n+s}+ \epsilon \omega_s. \label{pertuK}
\end{eqnarray}


If $\epsilon$ is small enough, by virtue of \eqref{utile} we can apply  Corollary \ref{corFJ} to the sets $K_{h_\lambda}, K_f, K_g$ and from \eqref{pertuK} 
we obtain that they satisfy assumption
\eqref{ipotcorFJ} with
\begin{eqnarray}
\delta= \frac{\epsilon \omega_s}{\mathcal{M}_{\frac{1}{n+s}}(\abs{K_f},\abs{K_g};\lambda)} =  
\frac{\epsilon}{\mathcal{M}_{\frac{1}{n+s}}(F,G;\lambda)}.\label{delta}
\end{eqnarray}

Then, if $\delta \leq e^{-M_{n+s}(\tau)}$, there exist a convex $K\subset \R^{n+s}$ and two homothetic copies $\hat{K}_f$ and  $\hat{K}_g$ of $K_f$ and $K_g$ such that  
\begin{equation}\label{utile2}
|\hat{K}_f|=|\hat{K}_g|=1, \qquad \pa{\hat{K}_f\cup\hat{K}_g} \subseteq K,
\end{equation}
and
\begin{equation}
   \abs{K\setminus \hat{K}_f} + \abs{K\setminus\hat{K}_g} \leq \tau^{-N_{n+s}} 
\pa{ \frac{\epsilon}{\mathcal{M}_{\frac{1}{n+s}}(F,G;\lambda)}}^{\sigma_{n+s}(\tau)}.\label{KFJdim} 
\end{equation}

\begin{oss}\label{rem1}
Since 
$|\hat{K}_f|=|\hat{K}_g|=1$,
\eqref{KFJdim} implies that the convex set $K$ has finite positive measure. Then it is bounded (since convex), whence 
\eqref{utile2} yields the boundedness of $K_f$ and $K_g$ which in turn implies the boundedness of the functions $f$ and $g$.
For simplicity, we can assume the convex $K$ is compact (possibly substituting it with its closure).
\end{oss}

In what follows, we indicate with $(x,y)\in \R^n\times \R^s$ an element of $\R^{n+s}$.
When we say (see just before \eqref{utile2}) that $\hat{K}_f$ and $\hat{K}_g$ are homothetic copies of $K_f$ and $K_g$, we mean 
that there exist $z_0=(x_0,y_0)\in\R^{n+s}$ and $z_1=(x_1,y_1)\in\R^{n+s}$
such that
\begin{equation}\label{hat}
\hat{K}_f=|K_f|^{-\frac{1}{n+s}}\left(K_f+z_0\right)\quad\text{and}\quad \hat{K}_g=|K_g|^{-\frac{1}{n+s}}\left(K_g+z_1\right)\,.
\end{equation}  Clearly, without loss of generality we can take $z_0=0$. 

To conclude the proof, we want now to show that, up to a suitable symmetrization, we can take $y_1=0$ (i.e. the translation of the homothetic copy $\hat{K}_g$ of $K_g$ is horizontal) and that the convex set $K$ given by Figalli and Jerison can be taken of the type $K_u$ 
for some $\frac{1}{s}$-concave function $u$.

For this, let us introduce the following Steiner type symmetrization in $\R^{n+s}$ with respect to the $n$-dimensional hyperspace
$\gra{y=0}$ (see for instance \cite{BuZa}). 
Let $C$ be a bounded measurable set in $\R^{n+s}$, for every $\bar{x}\in\R^n$ we set
$$
C(\bar{x})=C \cap \gra{x=\bx}=\{y\in\R^s\,:\, (\bar{x},y)\in C\}
$$
and 
\begin{equation}\label{maxraggio}
r_C(\bar{x})=\left(\omega_s^{-1}|C(\bar{x})|\right)^{1/s}\,.
\end{equation}
Then we define the $S$-symmetrand of $C$ as follows
\begin{eqnarray}
S(C)= \gra{\pa{\bar{x},y}\in \R^{n+s}:\ C\cap \gra{x=\bar{x}} \neq \emptyset,\ \abs{y}\leq r_C(\bar{x})}.  \label{simmesferica}
\end{eqnarray}
We notice that $S(C)$ is obtained
as union of the $s$-dimensional
closed balls of center $(\bx,0)$ and radius $r_C(\bx)$, for $\bx\in \R^n$ such that $C\cap \gra{x=\bx}$ is nonempty.
Thus, fixed $\bx$, the measure of the corresponding section of $S(C)$ is
\begin{equation}\label{stessamisura}
|S(C)\cap\gra{x=\bx}|=\omega_s r_C(\bx)^s  = |C(\bx)|\,.
\end{equation}

We describe the main properties of $S$-symmetrization, for 
bounded measurable susbsets of $\R^{n+s}$:\\
(i) if $C_1 \subseteq C_2$ then $S(C_1)\subseteq S(C_2)$ (obvious by definition);\\
(ii) $\abs{C}=\abs{S(C)}$ (consequence of \eqref{stessamisura} and Fubini's Theorem)
     so the S-symmetrization is measure preserving;\\
(iii) if $C$ is convex then $S(C)$ is convex (the proof is based
on the BM inequality in $\R^s$ and, for the sake of completeness, is given in the Appendix).\\


Now we symmetrize $K,\hat{K}_f,\hat{K}_g$ (and then replace them with $S(K),S(\hat{K}_f),S(\hat{K}_g)$).
Clearly \begin{equation}\label{simmf}
S(\hat{K}_f)= \hat{K}_f,
\end{equation}
\begin{equation}\label{simmg}
S(\hat{K}_g)=S\pa{|K_g|^{-\frac{1}{n+s}}(K_g +(x_1,y_1))}= |K_g|^{-\frac{1}{n+s}}(K_g +(x_1,0))\,.
\end{equation}
Moreover, (iii) implies that $S(K)$ is convex
and by (i) and \eqref{utile2} we have 
\begin{equation}\label{containing}
(S(\hat{K}_f)\cup S(\hat{K}_g))\subseteq S(K)\,.
\end{equation}
The latter, \eqref{KFJdim} and (ii) imply
\begin{equation}\label{KprimoFJ} \abs{S(K)\setminus S(\hat{K}_f)} + \abs{S(K)\setminus S(\hat{K}_g)} \leq \tau^{-N_{n+s}} 
\pa{ \frac{\epsilon}{\mathcal{M}_{\frac{1}{n+s}}(F,G;\lambda)}}^{\sigma_{n+s}(\tau)}\,.\end{equation}
Finally we notice that $S(K)$ is a compact convex set
of the desired form.

\begin{oss}
Consider the set $K_u$ associated to a function $u:\rn\to[0,+\infty)$ by \eqref{solido} and 
let $\bar{x}\in\rn$, $\bar{z}=(\bar{x},0)\in \R^{n+s},\ \mu>0$ and  
$$H=\mu\left(K_{u}+\bar{z}\right)\,.$$
Then 
$$
H=K_{v}
$$
(the set associated to $v$ by \eqref{solido}) where
\begin{equation}\label{isometrie}
v(x)=\mu^su\left(\frac{x-\bar{x}}{\mu}\right)\,.
\end{equation}
\end{oss}

From the previous remarks, we see that the sets $S(\hat{K}_f)$ and $S(\hat{K}_g)$ are in fact associated via \eqref{solido} to two functions 
$\hat{f}$ and $\hat{g}$, such that
\begin{equation}\label{hatKhatf}
S(\hat{K}_f)=K_{\hat{f}}\,,\quad S(\hat{K}_g)=K_{\hat{g}}\,,
\end{equation} 
and $\hat{f}$ and $\hat{g}$ are $s$-equivalent to $f$ and $g$ respectively, in the sense of \eqref{s-equivalent}, with 
\begin{equation}\label{mufmug}
\mu_f= \pa{\omega_s F}^{\frac{-1}{n+s}},\qquad
\mu_g=\pa{\omega_s G}^{\frac{-1}{n+s}}\,.
\end{equation}

We notice that the support sets $\Omega_0$ and $\Omega_1$  of $\hat{f}$ and $\hat{g}$ are given by
$$\Omega_0=\{x\in\rn\,:(x,0)\in S(\hat{K}_f)\}\,,\qquad \Omega_1=\{x\in\rn\,:\,(x,0)\in S(\hat{K}_g)\}$$
and that they are in fact homothetic copies of the support sets of the original functions $f$ and $g$.

Now we want to find a $\frac{1}{s}$-concave function $u$ such that $S(K)$ is associated to $u$ via \eqref{solido}.
We define $u: \R^n \longrightarrow [0,+\infty)$ as follows
\[ u(x) = \begin{cases}
    r_K(x)^s & \text{if $x\in \R^n:\ (x,0)\in S(K),$}\\
    0 & \text{otherwise}\,,
\end{cases}\]
and prove that
\begin{equation}\label{Ku} K_u = S(K)\,.\end{equation}

First notice that
\begin{eqnarray}
 \supp(u) =  \gra{x\in \R^n:\ (x,0)\in S(K)}. \label{suppinclu}\,
\end{eqnarray}
Indeed we have $\gra{z\in \R^n:\ u(z)>0} \subseteq \gra{x\in \R^n:\ (x,0)\in S(K)}$, 
whence $\supp(u)= \overline{\gra{z\in \R^n:\ u(z)>0}} \subseteq \gra{x\in \R^n:\ (x,0)\in S(K)}$, 
since the latter is closed.
Vice versa let $x$ such that $(x,0)\in S(K)$.
If $r_K(x)>0$ (see \eqref{maxraggio}) then $x \in \supp(u)$ obviously. 
Otherwise suppose $r_K(x)=0$, then, by the convexity of $S(K)$ and the fact that $S(K)$ is not contained in $\{y=0\}$, 
evidently 
\[\qua{\pa{U\setminus \gra{x}} \cap  \gra{z\in \R^n:\ r_K(z)>0}} \neq \emptyset\]
for every neighborhood $U$ of $x$,
i.e. $x \in \supp(u).$\\

By the definition of $u$ and \eqref{solido}, using \eqref{suppinclu}, 
we get
\[  K_u = \gra{(x,y)\in \R^n \times \R^s: \ x \in \supp(u),\ \abs{y}\leq u(x)^{1/s}}=\]
\[ = \gra{(x,y)\in \R^n \times \R^s: \ (x,0)\in S(K),\ \abs{y}\leq u(x)^{1/s}}=\]
\[ = \gra{(x,y)\in \R^n \times \R^s: \ (x,0)\in S(K),\ \abs{y}\leq r_K(x)}=S(K)\,.\]
Therefore we have shown
\eqref{Ku}
and from the convexity of $K$ follows that
$u$ is a $\frac{1}{s}$-concave function.
Being $K_u \supseteq \pa{K_{\hat{f}} \cup K_{\hat{g}}}$, clearly
\[ \supp(u)\supseteq \pa{\Omega_0 \cup \Omega_1}, \qquad  u \geq\hat{f} 
\ \text{in}\ \Omega_0, \qquad  \ u \geq \hat{g} \ \text{in}\ \Omega_1\,.\]
The final estimate can be deduced from \eqref{KprimoFJ}. Indeed, thanks to \eqref{volK}, we get
\[ \abs{K_u\setminus K_{\hat{f}}}= \abs{K_u}- \abs{K_{\hat{f}}}= 
\omega_s \int_{\R^n} (u-\hat{f}) \ dx,\]
and the same equality holds for $\abs{K_u\setminus K_{\hat{g}}}$. So \eqref{KprimoFJ} becomes 
\[ \int_{\R^n} (u-\hat{f}) \ dx \ + \ \int_{\R^n} \pa{u-\hat{g}} \ dx  \  \leq 
\omega_s^{-1}   \tau^{-N_{n+s}} 
\pa{ \frac{\epsilon}{\mathcal{M}_{\frac{1}{n+s}}(F,G;\lambda)}}^{\sigma_{n+s}(\tau)},\]
that is the desired result. 
\vspace{0.4 cm}

\section{A generalization to the case $s$ positive rational} 
We explain how Theorem \ref{stab1} can be generalized to a positive rational index $s$.
Given $f: \R^n \longrightarrow [0,+\infty)$ and an integer $q>0$,
we consider the auxiliary function $\tilde{f}: \R^{nq} \longrightarrow [0,+\infty)$
given by \eqref{fprod}, i.e.
\[\tilde{f}(x)= \tilde{f}(x_1,...,x_q)= \prod_{j=1}^q f(x_j),\] 
with $x= (x_1,...,x_q) \in \pa{\R^n}^q$.
Clearly $f$ is bounded if and only if $\tilde{f}$ is bounded.
We study further properties 
of functions of type \eqref{fprod}.

\begin{lemma}\label{eredconc}
Given an integer $q>0$, and a real $t>0$ 
let $\tilde{u}: \R^{nq} \longrightarrow [0,+\infty)$ 
be a function of the type \eqref{fprod}. Then $\tilde{u}$ is $t$-concave 
if and only if the 
function $u: \R^n \longrightarrow [0,+\infty)$ is $(qt)$-concave.
\end{lemma}
\begin{proof}
Suppose first that $\tilde{u}^t$ is concave. Fixed $\lambda\in (0,1),\ x,x'\in \R^n$,
we consider the element of $\R^{nq}$ which has
all the $q$ components identical to $(1-\lambda)x+\lambda x'$.
From hypothesis it holds
\[ \tilde{u}^{t}\pa{(1-\lambda)x+\lambda x',...,(1-\lambda)x+\lambda x'} \geq (1-\lambda)\tilde{u}^{t}\pa{x,...,x} 
+ \lambda \tilde{u}^{t}\pa{x',...,x'},\]
i.e. (thanks to \eqref{fprod})
\[ u^{qt}\pa{(1-\lambda) x+\lambda x'} \geq (1-\lambda)u^{qt}(x) + \lambda u^{qt}(x').\]
Thus $u^{qt}$ is concave.\\
Vice versa assume that $u^{qt}$ is concave, and fix 
$\lambda\in (0,1),\ x=(x_1,...,x_q),\ x'=\pa{x'_1,...,x'_q}  \in \pa{\R^n}^q$. 
We have
\[\tilde{u}^{t}\pa{(1-\lambda) x+\lambda x'} = \prod_{j=1}^q u^t\pa{(1-\lambda)x_j + \lambda x'_j}
= \prod_{j=1}^q \qua{u^{qt}\pa{(1-\lambda)x_j + \lambda x'_j}}^{1/q} \geq\]
\[ \geq \prod_{j=1}^q  \qua{(1-\lambda) u^{qt}(x_j) + \lambda u^{qt}(x'_j)}^{1/q} \geq 
\prod_{j=1}^q  (1-\lambda)^{1/q} u^{t}(x_j) +  \prod_{j=1}^q  \lambda^{1/q}  u^{t}(x'_j)=  \]
\[=   (1-\lambda) \prod_{j=1}^q  u^{t}(x_j) + \lambda \prod_{j=1}^q  u^{t}(x'_j)
= (1-\lambda)\tilde{u}^{t}(x) + \lambda \tilde{u}^{t}(x'),  \]
where the first inequality holds by concavity of $u^{qt}$, while
in the second one we have used Lemma \ref{Holder} with
$a_j = (1-\lambda)^{1/q} u^{t}(x_j), \ b_j=\lambda^{1/q}  u^{t}(x'_j)$.
Hence $u^t$ is concave.

\end{proof}
\begin{lemma}\label{maggprod}
Let $q>0$ integer and 
 $u \geq f\geq 0$ in $\rn$. 
Then \[\tilde{u}- \tilde{f} \ \geq  \ \widetilde{u-f}.\]
\end{lemma}
\begin{proof} 
The proof is by induction on the integer $q\geq 1$.  
The case $q=1$ is trivial, because in such case $\tilde{u}=u,\ \tilde{f}=f,\ \widetilde{u-f}= u-f$.
For the inductive step assume that the result is true until the index $q$, and
denote with $\tilde{\tilde{u}},
\tilde{\tilde{f}},\widetilde{\widetilde{u-f}}$ the respective functions of index $q+1$. By the definition \eqref{fprod}
\[ \pa{\tilde{\tilde{u}} - \tilde{\tilde{f}}}(x_1,...,x_{q+1})= 
\tilde{u}(x_1,...,x_q)u(x_{q+1}) - \tilde{f}(x_1,...,x_q)f(x_{q+1}),\]
\[ \widetilde{\widetilde{u-f}}(x_1,...,x_{q+1}) = \widetilde{u-f}(x_1,...,x_q)\cdot (u-f)(x_{q+1}).\]
These two equalities imply
\[ \pa{\tilde{\tilde{u}} - \tilde{\tilde{f}}}(x_1,...,x_{q+1})=\]
\[=\widetilde{\widetilde{u-f}}(x_1,...,x_{q+1}) - \widetilde{u-f}(x_1,...,x_q) \cdot \qua{u(x_{q+1})-f(x_{q+1})}+\]
\[+ \tilde{u}(x_1,...,x_q)u(x_{q+1}) - \tilde{f}(x_1,...,x_q)f(x_{q+1}) \geq \]
\[\geq  \widetilde{\widetilde{u-f}}(x_1,...,x_{q+1}) - \pa{\tilde{u}-\tilde{f}}(x_1,...,x_q)\qua{u(x_{q+1})-f(x_{q+1})}+\]
\[+ \tilde{u}(x_1,...,x_q)u(x_{q+1}) - \tilde{f}(x_1,...,x_q)f(x_{q+1}) =\]
\[= \widetilde{\widetilde{u-f}}(x_1,...,x_{q+1})+
f(x_{q+1})\qua{\tilde{u}(x_1,...,x_q) - \tilde{f}(x_1,...,x_q)}+\]
\[+\tilde{f}(x_1,...,x_q) \qua{ u(x_{q+1}) - f(x_{q+1})}\geq\]
\[ \geq \widetilde{\widetilde{u-f}}(x_1,...,x_{q+1}),\]
having used the inductive hypothesis and the assumption $u \geq f \geq 0$.
\end{proof}
\vspace{0.2 cm}

\begin{cor}\label{corstab}
Given an integer $n>0$, $\lambda \in (0,1),$
$s= \frac{p}{q}$ with $p,q$ positive integers, 
let  $f,\,g\in L^1(\rn)$ be nonnegative compactly supported functions such that
\[ F=\int_{\R^n} f\ dx>0 \quad \text{and}  \quad G=\int_{\R^n} g\ dx >0 .\]
Let $h: \R^n \longrightarrow [0,+\infty)$ satisfy assumption \eqref{ipotBBL} and
suppose there exists $\epsilon > 0 $ small enough such that
\begin{equation}\label{ipcorstab} 
\pa{\int_{\R^n} h \ dx}^q \leq 
\qua{\mathcal{M}_{\frac{1}{n+s}} \left(F, G\,; \lambda \right)}^q + \epsilon. 
\end{equation}

Then there exist a $\frac{1}{p}$-concave function $u': \R^{nq} \longrightarrow [0,+\infty)$ and two functions 
$\hat{f},\ \hat{g}: \R^{nq} \longrightarrow [0,+\infty)$, 
$p$-equivalent to $\tilde{f}$ and $\tilde{g}$ (given by \eqref{fprod}) in the sense of \eqref{s-equivalent} with 
\[ \mu_{\tilde{f}}= \omega_p^{\frac{-1}{nq+p}}  F^{\frac{-1}{n+s}},  \qquad
 \mu_{\tilde{g}}= \omega_p^{\frac{-1}{nq+p}}  G^{\frac{-1}{n+s}},\]
such that the following hold:
\[u' \geq \hat{f},
\qquad  \qquad  u' \geq \hat{g},\] 
\begin{eqnarray}  \int_{\R^{nq}} (u' -\hat{f}) dx +  \int_{\R^{nq}} \pa{u' -\hat{g}} dx \ \leq 
\  C_{nq+p}\pa{\frac{\epsilon}{\mathcal{M}_{\frac{1}{nq+p}} \left(F^q, G^q\,; \lambda \right)}}.  \label{tesistab2}
\end{eqnarray}
\end{cor}

\begin{proof}
We can assume $h= h_\lambda$.
Since $f$ and $g$ are nonnegative compactly supported functions belonging to $L^1(\rn)$, thus by \eqref{fprod} $\tilde{f},\tilde{g}$ are
nonnegative compactly supported functions belonging to $L^1(\R^{nq})$.
The assumption \eqref{ipcorstab} is equivalent, considering the corresponding functions
$\tilde{f},\tilde{g},\tilde{h}: \R^{nq} \longrightarrow [0,+\infty)$ and using \eqref{intprod}, to
\[\int_{\R^{nq}} \tilde{h}\ dx \leq 
\qua{(1-\lambda) \pa{\int_{\R^{nq}} \tilde{f}  \ dx}^{\frac{1}{nq+qs}} +
\lambda \pa{\int_{\R^{nq}} \tilde{g}\ dx}^{\frac{1}{nq+qs}}}^{nq+qs} + \epsilon \]
\begin{equation} \text{i.e.} \qquad \qquad \int_{\R^{nq}} \tilde{h}\ dx \leq 
{\mathcal{M}_{\frac{1}{nq+p}} \left(F^q, G^q\,; \lambda \right)} + \epsilon. \label{stab2}
\end{equation}
We notice that the index $qs=p$ is integer, while $nq$ is exactly the dimension of the space in 
which $\tilde{f},\tilde{g},\tilde{h}$ are defined.\\
To apply Theorem \ref{stab1}, we have to verify that $\tilde{f},\tilde{g},\tilde{h}$ satisfy
the corresponding inequality \eqref{ipotBBL} of index $qs$. Given $x_1,...,x_q \in\supp(f),\ x'_1,...,x'_q \in\supp(g)$, 
let $x= (x_1,...,x_q)\in \supp(\tilde{f}),\ x'= (x'_1,...,x'_q) \in \supp(\tilde{g})$.
By hypothesis, we know that $f,g,h$ satisfy \eqref{ipotBBL}, 
in particular for every $j=1,...,q$
\[h\pa{(1-\lambda)x_j + \lambda x'_j} \geq \qua{(1-\lambda)f(x_j)^{1/s} + \lambda g(x'_j)^{1/s}}^s.\]
This implies
\[\prod_{j=1}^q h\pa{(1-\lambda)x_j + \lambda x'_j}
\geq  \qua{\prod_{j=1}^q \qua{(1-\lambda)f(x_j)^{1/s} + \lambda g(x'_j)^{1/s}}}^s \geq \]
\begin{eqnarray}
\geq  \qua{(1-\lambda)\pa{\prod_{j=1}^q f(x_j)}^{1/qs} + 
\lambda\pa{\prod_{j=1}^q g(x'_j)}^{1/qs}}^{qs},  \label{hold} 
\end{eqnarray}
where the last inequality is due to Corollary \ref{corHolder}.
By definition of \eqref{fprod}, \eqref{hold} means that for every $x\in \supp(\tilde{f}),\ x'\in \supp(\tilde{g})$ we have
\[ \tilde{h}\pa{(1-\lambda)x+ \lambda x'} \geq 
\qua{(1-\lambda)\tilde{f}(x)^{1/qs} + \lambda\tilde{g}(x')^{1/qs}}^{qs},\]
i.e. the functions  $\tilde{f},\tilde{g},\tilde{h}: \R^{nq} \longrightarrow [0,+\infty)$
satisfy the hypothesis \eqref{ipotBBL} with the required index $qs$.
Therefore we can apply Theorem \ref{stab1} and conclude that there exist
a $\frac{1}{p}$-concave function $u': \R^{nq} \longrightarrow [0,+\infty)$ and 
two functions $\hat{f}, \hat{g},$ 
$p$-equivalent to $\tilde{f}$ and $\tilde{g}$,
with the required properties.
The estimate \eqref{tesistab}, applied to \eqref{stab2}, implies
\[\int_{\R^{nq}} (u' -\hat{f}) \ dx \ + \ \int_{\R^{nq}} \pa{u' -\hat{g} } \ dx 
\ \leq   \  C_{nq+p}\pa{\frac{\epsilon}{\mathcal{M}_{\frac{1}{nq+p}} \left(F^q, G^q\,; \lambda \right)}}.\] 
\end{proof}

\begin{oss}
Assume $F=G$ and, for simplicity,
suppose that $\hat{f}=\tilde{f},\ \hat{g}=\tilde{g}$ in Corollary \ref{corstab}
(as it is true up to a p-equivalence). Moreover
assume that the $\frac{1}{p}$-concave function $u': \R^{nq} \longrightarrow [0,+\infty)$,
given by Corollary \ref{corstab}, is of the type \eqref{fprod}, i.e. $u'= \tilde{u}$
where $u: \R^{n} \longrightarrow [0,+\infty)$ has to be 
$\frac{1}{s}$-concave by Lemma \ref{eredconc}.
In this case Corollary \ref{corstab} assumes a simpler statement, which naturally
extends the result of Theorem \ref{stab1}.
Indeed \eqref{tesistab2}, thanks to Lemma \ref{maggprod}, becomes
\[ \int_{\R^{nq}} \widetilde{u-f} \ dx \ + \ \int_{\R^{nq}} \widetilde{u-g} \ dx 
\ \leq   \  C_{nq+p}\pa{\frac{\epsilon}{\mathcal{M}_{\frac{1}{nq+p}} \left(F^q, G^q\,; \lambda \right)}}, \qquad \text{i.e.}\]
\begin{equation}\label{corsemplice}
\qua{\int_{\R^{n}} \pa{u-f} \ dx}^q \ + \ \qua{\int_{\R^{n}} \pa{u-g} \ dx}^q \ 
\leq   \  C_{nq+p}\pa{\frac{\epsilon}{\mathcal{M}_{\frac{1}{nq+p}} \left(F^q, G^q\,; \lambda \right)}}. 
\end{equation} 


Unfortunately the function $u'$ constructed in Theorem \ref{stab1} is not necessarely of the desired form, that is in general
we can not find a function $u: \R^{n} \longrightarrow [0,+\infty)$ such that $u'= \tilde{u}$ (a counterexample can be 
explicitly given). Then our proof can not be easily extended to the general case $s \in \Q$ to get \eqref{corsemplice}.

\end{oss}
\vspace{0.2 cm}

\section{A stability for $s>0$}
To complete the paper, we give a (weaker) version of our main stability result Theorem \ref{stab1}
which works for an arbitrary real index $s>0$. For this, let us denote by 
$[s]$ the integer part of $s$, i.e. the largest integer not greater than $s$.
Obviously $[s]+1 > s\geq [s]$, whereby 
(by the monotonicity of $p$-means with respect to $p$, i.e.
$\mathcal{M}_p(a,b; \lambda) \leq \mathcal{M}_q(a,b; \lambda)$ if $p \leq q$) 
for every $a,b \geq 0,\ \lambda \in (0,1)$
\begin{eqnarray}
\qua{(1-\lambda)a^{\frac{1}{s}} + \lambda b^{\frac{1}{s}}}^s \geq  
\qua{(1-\lambda)a^{\frac{1}{[s]+1}} + \lambda b^{\frac{1}{[s]+1}}}^{[s]+1}, \label{Ms}
\end{eqnarray}
\begin{eqnarray}
\qua{(1-\lambda)a^{\frac{1}{n+s}} + \lambda b^{\frac{1}{n+s}}}^{n+s} \geq  
\qua{(1-\lambda)a^{\frac{1}{n+[s]+1}} + \lambda b^{\frac{1}{n+[s]+1}}}^{n+[s]+1}. \label{Mn+s}
\end{eqnarray}
We arrive to the following corollary for every index $s>0$.

\begin{cor}\label{stabparteintera}
Given $s>0,\ \lambda \in (0,1),$ let 
$f,g: \R^n \longrightarrow [0,+\infty)$ be integrable functions 
such that
\begin{equation}\label{normalization}
\int_{\R^n} f\ dx=\int_{\R^n} g\ dx =1\,.\end{equation}
Assume $h: \R^n \longrightarrow [0,+\infty)$ satisfies assumption \eqref{ipotBBL2} and
there exists $\epsilon > 0 $ small enough such that
\begin{equation}  
\int_{\R^n} h\ dx \leq 
1 + \epsilon.\label{ipotstabparteintera}
\end{equation}
Then there exist a $\frac{1}{[s]+1}$-concave function $u: \R^n \longrightarrow [0,+\infty)$  and two functions $\hat{f}$ and $\hat{g}$, $([s]+1)$-equivalent to $f$ and $g$ in the sense of \eqref{isometrie}
(with $\mu_f= \mu_g=\pa{\omega_{[s]+1} }^{\frac{-1}{n+[s]+1}}$) 
such that
\[ u \geq \hat{f}, \qquad  \qquad u \geq \hat{g},\] and
\[\int_{\R^n} (u-\hat{f}) \ dx \ + \ \int_{\R^n} \pa{u-\hat{g}} \ dx  \  \leq \ C_{n+[s]+1}(\epsilon).\]
\end{cor}

\vspace{0.3 cm}
\begin{proof}
We notice that the assumption \eqref{ipotBBL2} (i.e. the hypothesis of BBL of index $\frac{1}{s}$),
through \eqref{Ms}, implies that for every $x_0 \in\supp(f),\ x_1 \in\supp(g)$ 
\[h\pa{(1-\lambda)x_0 + \lambda x_1} \geq \qua{(1-\lambda)f(x_0)^{\frac{1}{[s]+1}}
 + \lambda g(x_1)^{\frac{1}{[s]+1}}}^{[s]+1},\]
i.e. the corresponding hypothesis of BBL for the index $\frac{1}{[s]+1}$.
Therefore, thanks to the assumptions \eqref{normalization} and \eqref{ipotstabparteintera}, it holds
$ \int h  \leq  1 + \epsilon = \mathcal{M}_{\frac{1}{n+[s]+1}}(\int f,\int g;\lambda)+\epsilon, $
so we can apply directly Theorem \ref{stab1} using the integer $[s]+1$ as index.
This concludes the proof.
\end{proof}
\vspace{0.3 cm}
\begin{oss}
If we don't use the normalization \eqref{normalization} and want to write a result for generic unrelated $F=\int f$ and $G=\int g$, we can notice that
assumption \eqref{ipotstabparteintera} should be replaced by
$$
\int_{\R^n} h\ dx \leq \mathcal{M}_{\frac{1}{n+[s]+1}}(F,G;\lambda)+\epsilon\,.
$$
On the other hand, thanks to assumption \eqref{ipotBBL2},
we can apply Proposition \ref{BBL2} and obtain
\[ \int_{\R^n} h\ dx \geq  
\mathcal{M}_{\frac{1}{n+s}}(F,G;\lambda).\]
Then we would have
$$
\mathcal{M}_{\frac{1}{n+s}}(F,G;\lambda)\leq \mathcal{M}_{\frac{1}{n+[s]+1}}(F,G;\lambda)+\epsilon\,.
 $$
The latter inequality is possible only if $F$ and $G$ are close to each others, thanks to the stability of the monotonicity property of $p$-means, which states
$$
\mathcal{M}_{\frac{1}{n+[s]+1}}(F,G;\lambda)\leq\mathcal{M}_{\frac{1}{n+s}}(F,G;\lambda),
$$
with equality if and only if $F=G$. 
In this sense the normalization \eqref{normalization} cannot be completely avoided and the result obtained in Corollary \ref{stabparteintera} is weaker than what desired. Indeed notice in particular that it does not coincide with Theorem \ref{stab1} even in the case 
when $s$ is integer, since $[s]+1>s$ in that case as well.
\end{oss}
\vspace{0.3 cm}

\section{Appendix}
Here we show that the $S$-symmetrization, introduced in Remark \ref{rem1}, preserves the convexity of the involved set
(that is the property (iii) therein).\\
We use the notations of Remark \ref{rem1}, in particular we refer to \eqref{maxraggio} and \eqref{simmesferica}, 
and remember that $C$ is a bounded measurable set in $\R^{n+s}$. We need the following preliminary result,
based on the Brunn-Minkowski inequality in $\R^s$.



\begin{lemma}\label{disraggi}
If $C$ is convex, then for every 
$t\in (0,1)$ and every $x_0,x_1\in \R^n$ 
such that 
$C(x_0), C(x_1)$ are nonempty sets, it holds
\begin{eqnarray}  
(1-t)r_C(x_0) + t r_C(x_1)\ \leq\ r_C((1-t)x_0 + t x_1).  \label{raggi}
\end{eqnarray}
\end{lemma}



\begin{proof}
By defintion of \eqref{maxraggio} \[r_C(x_0)= {\omega_s}^{-1/s} |C(x_0)|^{1/s},\quad r_C(x_1)= {\omega_s}^{-1/s} |C(x_1)|^{1/s},\] thus 
\begin{eqnarray}  
(1-t)r_C(x_0) + t r_C(x_1)=  \omega_s^{-1/s} \qua{(1-t) |C(x_0)|^{1/s} + t |C(x_1)|^{1/s}}. \label{concraggi}
\end{eqnarray}
Since $C$ is convex, we notice that $C(x_0), C(x_1)$ are (nonempty) convex sets in $\R^s$ such that
\begin{equation}\label{disugC}
(1-t) C(x_0) + t C(x_1) \ \subseteq\ C((1-t)x_0 + t x_1).
\end{equation}
Applying BM inequality (i.e. Proposition \ref{BM}) to the sets $C(x_0), C(x_1) \subset \R^s$, \eqref{concraggi} implies
\[ (1-t)r_C(x_0) + t r_C(x_1) \ \leq\ \omega_s^{-1/s} \abs{(1-t) C(x_0) + t C(x_1)}^{1/s} \leq \]
\[ \leq \omega_s^{-1/s} \abs{C((1-t)x_0 + t x_1)}^{1/s} = r_C((1-t)x_0 + t x_1), \]
where in the last inequality we use \eqref{disugC}.
\end{proof}

\begin{prop}\label{simmconv}
If $C$ is convex then $S(C)$ is convex.
 \end{prop}
\begin{proof}
Let $t \in (0,1)$, and let $P=(x_0,y_0),Q=(x_1,y_1)$ be two distinct points belonging to $S(C)$, i.e. $C(x_0), C(x_1)$ are nonempty sets 
and
\begin{eqnarray}
\abs{y_0}\leq r_C(x_0), \qquad  \abs{y_1}\leq r_C(x_1). \label{lemma1}
\end{eqnarray}
We prove that \[(1-t)P+tQ = \pa{(1-t)x_0+tx_1, (1-t)y_0+ty_1} \in S(C).\]
By assumptions and \eqref{disugC} the set $C((1-t)x_0 + t x_1)$ is nonempty. Furthermore by the triangle inequality,
 \eqref{lemma1} and Lemma \ref{disraggi} we obtain
\[ \abs{(1-t)y_0 + ty_1} \leq (1-t)\abs{y_0} + t \abs{y_1} \leq (1-t)r_C(x_0) + t r_C(x_1) \leq r_C((1-t)x_0 + t x_1).\]
Then $(1-t)P+tQ\in S(C)$, i.e. $S(C)$ is convex.
\end{proof}


\vspace{0.3 cm}

\pagestyle{plain}

\end{document}